\documentclass[11pt,a4paper]{article}
\usepackage{caption}
\usepackage{epsf,epsfig,amsfonts,amsgen,amsmath,amstext,amsbsy,amsopn,amsthm}
\usepackage{color}

\newtheorem{thm}{Theorem}[section]
\newtheorem{cor}[thm]{Corollary}
\newtheorem{lem}[thm]{Lemma}

\newtheorem{conj}[thm]{Conjecture}

\newtheorem{claim}{Claim}

\newcommand{\pf}{\noindent{\it Proof.} }

\newcommand{\bin}{\operatorname{Bin}}
\newcommand{\Deg}{\operatorname{Deg}}

\begin{document}

\title{On the rainbow matching conjecture for 3-uniform hypergraphs}
\author{
	Jun Gao\thanks{School of Mathematical Sciences, USTC, Hefei, Anhui 230026, China. Email: gj0211@mail.ustc.edu.cn.}~~~~~~~	
	Hongliang Lu\thanks{School of Mathematics and Statistics, Xi'an Jiaotong University, Xi'an, Shaanxi 710049, China. Partially supported by NSFC grant 11871391 and Fundamental Research Funds for the Central Universities. Email: luhongliang@mail.xjtu.edu.cn.}~~~~~~~
	Jie Ma\thanks{School of Mathematical Sciences, USTC, Hefei, Anhui 230026, China. Research supported by the National Key R and D Program of China 2020YFA0713100, National Natural Science Foundation of China grants 11622110 and 12125106, and Anhui Initiative in Quantum Information Technologies grant AHY150200. Email: jiema@ustc.edu.cn.}~~~~~~~
    Xingxing Yu\thanks{School of Mathematics, Georgia Institute of Technology, Atlanta, GA 30332, USA. Partially supported by NSF grant DMS-1954134. Email: yu@math.gatech.edu.}~~~~~
}

\date{}

\maketitle

\begin{abstract}
Aharoni and Howard, and, independently, Huang, Loh, and Sudakov proposed the following rainbow version of Erd\H{o}s matching conjecture:
For positive integers $n,k,m$ with $n\ge km$, if each of the families $F_1,\ldots, F_m\subseteq {[n]\choose k}$ has size more than $\max\{\binom{n}{k} - \binom{n-m+1}{k}, \binom{km-1}{k}\}$,
then there exist pairwise disjoint subsets $e_1,\dots, e_m$ such that $e_i\in F_i$ for all $i\in [m]$.
We prove that there exists an absolute constant $n_0$ such that this rainbow version holds for $k=3$ and $n\geq n_0$. We convert this rainbow matching problem to a matching problem on a special hypergraph $H$.
We then combine several existing techniques on matchings in uniform hypergraphs: find an absorbing matching $M$ in $H$;
use a randomization process of Alon et al. to find an almost regular subgraph of $H-V(M)$; and find an almost perfect matching in $H-V(M)$.
To complete the process, we also need to prove a new result on matchings in 3-uniform hypergraphs,
which can be viewed as a stability version of a result of {\L}uczak and Mieczkowska and might be of independent interest.
\end{abstract}

{{\bf Key words:}  Rainbow matching conjecture, Erd\H{o}s matching conjecture, Stability

\vskip 0.3cm


\vskip 0.3cm

\section{Introduction}
For a positive integer $k$ and a set $V$, let $[k]:=\{1,...,k \}$ and $\binom{V}{k}:=\{A\subseteq V: |A|=k\}$.	
A hypergraph $H$ consists of a vertex set $V(H)$ and an edge set $E(H)\subseteq 2^{V(H)}$.
A hypergraph $H$ is \emph{$k$-uniform} if all its edges have size $k$ and we call it a \emph{$k$-graph} for short.
Throughout this paper, we often identify $E(H)$ with $H$ when there is
no confusion and, in particular, denote by $|H|$ the number of edges
in $H$. Given a set $T$ of edges in $H$, we use $V(T)$ to denote $\bigcup _{e\in T} e$.
Given a vertex subset $S\subseteq V(H)$ in $H$, we use $H[S]$ to denote the subgraph of $H$ induced by $S$, and let $H-S = H[V(H)\setminus S]$.

A \emph{matching} in a hypergraph $H$ is a set of pairwise disjoint edges in $H$.
We use $\nu(H)$ to denote the maximum size of a matching in $H$.
Let $\mathcal{F}= \{F_1, ... ,F_m\}$ be a family of hypergraphs on the same vertex set.
A set of $m$ pairwise disjoint edges is called a \emph{rainbow matching} for $\mathcal{F}$ if each edge is from a different $F_i$.
If such a matching exists, then we also say that $\mathcal{F}$ {\it admits a rainbow matching}.

A classical problem in extremal set theory asks for the maximum number of edges in $n$-vertex $k$-graphs $H$ with $\nu(H)<m$.
Let $n,k,m$ be positive integers with $n\geq km$. The $k$-graphs $S(n,m,k):=\binom{[n]}{k}\backslash
\binom{[n]\backslash [m-1]}{k}$ and  $D(n,m,k):= \binom{[km-1]}{k}$ on the same vertex set $[n]$ do
not have matchings of size $m$.
Erd\H{o}s \cite{Erdos65} conjectured in 1965 that among all $k$-graphs
with no matching of size $m$, $S(n,m,k)$ or $D(n,m,k)$ has the maximum number of edges:
Any $n$-vertex $k$-graph $H$ with $\nu(H)<m$ contains at most $$f(n,m,k):= \max \left\{ {n\choose k}-{n-m+1\choose k}, {km-1\choose k}\right\}$$ edges.
This is often referred to as the {\it Erd\H{o}s matching conjecture} in
the literature, and there has been extensive research on this
conjecture, see, for instance, \cite{AFH12,BDE76,Fr13,F17,Fr172,FLM12,FK18,LM14}.
In particular, the special case for $k=3$ was settled for large $n$ by {\L}uczak and Mieczkowska \cite{LM14} and
completely resolved by Frankl \cite{F17}.

The following analogous conjecture, known as the {\it
  rainbow matching conjecture}, was made by Aharoni and Howard
\cite{ADP} and, independently, by Huang, Loh, and Sudakov
\cite{HLS12}. For related topics on rainbow type problems, we refer the interested reader to \cite{JK20,KK20,LYY20,Py86}.

\begin{conj}[\cite{ADP,HLS12}]\label{RMConj}
Let $n, k, m$ be positive integers with $n\geq km$.
Let $\mathcal{F}= \{F_1, ... ,F_m\}$ be a family of $k$-graphs on the same vertex set $[n]$ such that $|F_i|> f(n,m,k)$ for all $i\in [m]$.
Then $\mathcal{F}$ admits a rainbow matching.
\end{conj}

The case $k=2$ of this conjecture is in fact a direct consequence of an earlier result of Akiyama and Frankl \cite{AF} (which was restated \cite{F87}).
The following was obtained by Huang, Loh, and Sudakov \cite{HLS12}.
\begin{thm}[\cite{HLS12}, Theorem~3.3]\label{n>k^2}
Conjecture~\ref{RMConj} holds when $n>3k^2m$.
\end{thm}

Keller and Lifshitz \cite{KL17} proved that Conjecture~\ref{RMConj} holds when $n\geq f(m)k$ for some large constant $f(m)$ which only depends on $m$,
and this was further improved to $n=\Omega(m\log m)k$ by Frankl and Kupavskii \cite{FK20}.
Both proofs use the junta method.
Very recently, Lu, Wang, and Yu \cite{LWY20} showed that Conjecture~\ref{RMConj} holds when $n\geq 2km$ and $n$ is sufficiently large.

The following is our main result, which proves Conjecture~\ref{RMConj} for $k=3$ and sufficiently large $n$.

\begin{thm}\label{main}
There exists an absolute constant $n_0$ such that the following holds for all $n\geq n_0$.
For any positive integers $n, m$ with $n\geq 3m$, let $\mathcal{F}= \{F_1, ... ,F_m\}$ be a family of $3$-graphs on the same vertex set $[n]$ such that $|F_i|> f(n,m,3)$ for all $i\in [m]$.
Then $\mathcal{F}$ admits a rainbow matching.
\end{thm}

Our proof of Theorem~\ref{main} uses some new ideas and combines different techniques from
Alon-Frankl-Huang-R\H{o}dl-Rucinski-Sudakov \cite{AFH12},
{\L}uczak-Mieczkowska \cite{LM14}, and Lu-Yu-Yuan \cite{LYY19}.
(For a high level description of our proof, we refer the reader to Section 2 and/or Section 7.)
In the process, we prove a stability result on 3-graphs (see Lemma~\ref{stability}) that plays a crucial role in
our proof and might be of independent interest: If the number of edges in an $n$-vertex 3-graph $H$ with $\nu(H)<m$ is close to $f(n,m,3)$,
then $H$ must be close to $S(n,m,3)$ or $D(n,m,3)$.

The rest of the paper is organized as follows.
In Section~2, we introduce additional notation,
and state and/or prove a few lemmas for later use. 
In Section~3,  we deal with the families $\mathcal{F}$ in which most $3$-graphs are close to the same $3$-graph that is $S(n,m,3)$ or $D(n,m,3)$.
To deal with the remaining families,
we need the above mentioned stability result for matchings in 3-graphs, which is done in Section 4.
In Section~5, we show that there exists an absolute constant $c>0$
such that Theorem~\ref{main} holds for $m>(1-c)n/3$.
The proof of Theorem~\ref{main} for $m\le(1-c)n/3$ is completed in Section~6.
Finally, we complete the proof of Theorem~\ref{main} in Section 7.

\section{Previous results and lemmas}
In this section, we define saturated families and stable hypergraphs,
and state several lemmas that we will use frequently.
We begin with some notation.
Suppose that $H$ is a hypergraph and $U, T$ are subsets of $V(H)$.
Let $N_H(T) := \{A: A\subseteq V(H)\setminus T \mbox{ and } A\cup T \in E(H)\}$ be the {\it neighborhood} of $T$ in $H$, and let $d_H(T) := |N_H(T)|$.
We write $d_H(v)$ for $d_H(\{v\})$.
Let $\Delta(H) :=\max_{v\in V(H)}d_H(v)$ and $\Delta_2(H):=\max_{T\in \binom{V(H)}{2}}d_H(T)$.
In case $T\subseteq U$, we often identify $d_{H[U]}(T)$ with $d_{U}(T)$ when there is no confusion.

It will be helpful to consider ``maximal'' counterexamples to
Conjecture~\ref{RMConj}.
Let $n, k, m$ be positive integers with $n\geq km$ and let $\mathcal{F} = \{F_1,...,F_m\}$ be a family of $k$-graphs on the same vertex set $[n]$.
We say that $\mathcal{F}$ is {\bf saturated}, if $\mathcal{F}$ does not admit a rainbow matching, but for every $F\in \mathcal{F}$ and $e\notin F$,
the new family $\mathcal{F}(e,F):=(\mathcal{F}\backslash \{F\})\cup
\{F\cup \{e\}\}$ admits a rainbow matching.
The following lemma says that the vertex degrees of every $k$-graph in a saturated family are typically small.

\begin{lem}\label{claim:Delta}
Let $n,k,m$ be positive integers with $n\ge km$.
Let $\mathcal{F}= \{{F}_1, ... ,{F}_m\}$ be a saturated family of $k$-graphs on the same vertex set $[n]$.
Then for each $v\in [n]$ and each $i\in [m]$,   $d_{F_i}(v)\le
\binom{n-1}{k-1} - \binom{n-1-k(m-1)}{k-1}$ or $d_{F_i}(v) = \binom{n-1}{k-1}$.
\end{lem}

\pf
Suppose $d_{F_i}(v) <\binom{n-1}{k-1}$, where $v\in [n]$ and $i\in [m]$.
Then there exists $e\in {[n]\choose k}\setminus F_i $ such that $v\in e$.
Since $\mathcal{F}$ is saturated, the family ${\cal F}(e,F_i)$ admits
a rainbow matching, say $M\cup \{e\}$, with $M$ being a
rainbow matching for the family $\mathcal{F}\setminus \{F_i\}$.

If $d_{F_i}(v) > \binom{n-1}{k-1} -
\binom{n-1-k(m-1)}{k-1}=\left|\binom{[n]\backslash \{v\}}{k-1}
  \backslash \binom{[n]\backslash (\{v\}\cup V(M))}{k-1}\right|$, then
there exists an edge $f\in F_i$ such that $v\in f$ and  $f\cap
V(M)=\emptyset$. Now $M\cup \{f\}$ is a rainbow matching for
$\mathcal{F}$, a contradiction.
So $d_{F_i}(v)\le
\binom{n-1}{k-1} - \binom{n-1-k(m-1)}{k-1}$.
\qed

\medskip

We will remove vertices of degree $\binom{n-1}{k-1}$ and
use Lemma~\ref{claim:Delta} to produce saturated family $\mathcal{F}=
\{{F}_1, ... ,{F}_m\}$ of $k$-graphs such that  for each $v\in V(F_i)$ and each $i\in [m]$,   $d_{F_i}(v)\le
\binom{n-1}{k-1} - \binom{n-1-k(m-1)}{k-1}$.

Next we define stable hypergraphs.
Let $n,k$ be positive integers with $n\ge k$. Let $e=\{a_1,...,a_k\}$
and $f=\{b_1,...,b_k\}$ be members of $\binom{[n]}{k}$ with
$a_1<a_2<...<a_k$ and $b_1<b_2<...<b_k$.
We write $e\le f$ if $a_i\le b_i$ for all $1\le i \le k$, and $e<f$ if $e\le f$ and $e\ne f$.

A $k$-graph $F\subseteq \binom{[n]}{k}$ is said to be {\bf stable} if $e<f\in F$
implies $e\in F$. A family ${\cal F}$ of $k$-graphs on the same vertex
set $[n]$ is {\bf stable} if each $k$-graph in ${\cal F}$ is stable.

The following result of Huang, Loh, and Sudakov \cite{HLS12} will be
used frequently, which enables us
to work with stable families when proving Conjecture~\ref{RMConj}.

\begin{lem}[\cite{HLS12}, Lemma 2.1]\label{lem:stable}
Let $n,k,m$ be positive integers with $n\ge km$.
If the family $\{F_1, ..., F_m\}$ of $k$-graphs with $V(F_i) = [n]$ for all $i\in [m]$ does not admit a rainbow matching,
then there exists a stable family $\{F_1', ..., F_m'\}$ of $k$-graphs with $|F_i| = |F_i'|$ and $V(F_i') = [n]$ for all $i\in [m]$ which still preserves this property.
\end{lem}

\begin{cor} \label{coro:stable-saturated}
Let $n,k,m$ be positive integers with $n\ge km$.
Let ${\cal F}=\{F_1, \ldots, F_m\}$ be a family of $k$-graphs on the vertex set $[n]$ that
does not admit a rainbow matching. Then there exists a family ${\cal
  F}'=\{F_1', \ldots, F_m'\}$ of $k$-graphs on the same vertex set $[n]$ such that ${\cal F}'$ is both stable and
saturated and $|F_i'|\ge
|F_i|$ for $i\in [m]$.
\end{cor}
\pf  Let  ${\cal  F}^*=\{F_1^*, \ldots, F_m^*\}$ be a family of
$k$-graphs on the same vertex set $[n]$ such that ${\cal F}^*$ admits no rainbow matching,  $|F_i^*|\ge
|F_i|$ for $i\in [m]$, and, subject to these, $\sum_{i\in [m]}|F_i^*|$
is maximum.

Then ${\cal  F}^*$ is saturated. Now apply
Lemma~\ref{lem:stable} to ${\cal F}^*$ we obtain a stable family ${\cal F}'
=\{F_1', \ldots, F_m'\}$ of $k$-graphs on the vertex set $[n]$ such that ${\cal F}'$ admits no rainbow
matching,  and $|F_i'|=|F_i^*|$ for $i\in [m]$. By the choice of
${\cal F}^*$, we see that ${\cal F}'$ is also saturated. \qed

\medskip

We now describe an operation that converts a rainbow matching problem to
a matching problem on a single hypergraph. Let $n, k, m, r$ be
non-negative integers, with $r=\lfloor n/k\rfloor-m$ and $m\ge 1$.
Let ${\cal F}=\{F_1, \ldots, F_m\}$ be a family of $k$-graphs on
the same vertex set $[n]$, and
let $\mathcal{V}= \{v_1,...,v_m\}$ and $\mathcal{U}= \{u_1,...,u_r\}$ be two disjoint sets such that $(\mathcal{V}\cup \mathcal{U})\cap [n]=\emptyset$.
We use $H(\mathcal{F})$ to denote the $(k+1)$-graph with  vertex set
$[n]\cup \mathcal{V}$ and  edge set $\bigcup_{i=1}^{m} \{e\cup\{v_i\}\
:\ e\in F_i\}$, and use $H^*(\mathcal{F})$ to denote the $(k+1)$-graph
with the vertex set $[n]\cup \mathcal{V}\cup \mathcal{U}$ and the edge
set $E(H(\mathcal{F})) \cup \bigcup_{i=1}^{r}\{e\cup \{u_i\}\ :\ e\in
\binom{[n]}{k}\}$. If $F_1 = ... = F_m = S(n,m,k)$ (respectively, $F_1 = ... = F_m =D(n,m,k)$), then we write $H(\mathcal{F})$ as $H_S(n,m,k)$ (respectively, $H_D(n,m,k)$).

It is easy to see that $\mathcal{F}$ admits a rainbow matching
if and only if $H(\mathcal{F})$ has a matching of size $m$,
which is also if and only if $H^*(\mathcal{F})$ has a matching of size $m+r$.
This allows us to access existing approaches
and tools invented for matching problems. For instance, we
take the  approach by considering  whether
or not the hypergraphs $H(\mathcal{F})$ in question are close to the extremal
configurations $H_S(n,m,k)$ and $H_D(n,m,k)$. We will see in Section~\ref{Sec:HD} that if $H({\cal F})$ is close to $H_D(n,m,k)$
and ${\cal F}$ is stable, then
${\cal F}$ admits a rainbow matching.

Here we give an easy lemma concerning a case when  $H({\cal F})$ is
not close to $H_S(n,m,k)$, which will be used along with Lemma~\ref
{claim:Delta}. Let $H_1$ and $H_2$ be two $k$-graphs on the same vertex set $V$ and let $\epsilon$ be some positive real;
we say that $H_2$ is {\bf $\epsilon$-close} to $H_1$ if $|E(H_1) \setminus E(H_2)|\le \epsilon|V|^k$.

\begin{lem}\label{lemma:H_S(n,m)}
For any given integer $k\geq 3$, let $\epsilon, c$ be reals such that $0<\epsilon \ll c\ll 1$.\footnote{Here and throughout the rest of the paper, the notation $a\ll b$ means that $a$ is sufficiently small compared with $b$ which need to satisfy finitely many inequalities in the proof.}
Let $n, m$ be integers such that $n/3k^2\le m\le(1-c)n/k$.
Let $\mathcal{F}=\{F_1,...,F_m \}$ be a family of $k$-graphs on vertex set $[n]$.
If for every $i\in [m]$ and $v\in [n]$, $d_{F_i}(v)\le \binom{n-1}{k-1} - \binom{n-k(m-1)-1}{k-1}$,
then $H(\mathcal{F})$ is not $\epsilon$-close to $H_S(n,m,k)$.
\end{lem}

\begin{proof}
We note that $S(n,m,k)$ has $m-1$ vertices of degree $\binom{n-1}{k-1}$.
Since for every $i\in [m]$ and $v\in [n]$, $d_{F_i}(v)\le  \binom{n-1}{k-1} - \binom{n-k(m-1)-1}{k-1}$, we have
\begin{align*}
|E(H_S(n,m,k))\setminus E(H(\mathcal{F}))|\geq m\cdot (m-1) \cdot \binom{n-k(m-1)-1}{k-1}\cdot \frac1k > \frac{n^2}{10k^5} \binom{cn}{k-1} > \epsilon (n+m)^{k+1},
\end{align*}
where the second inequality is due to  $n/3k^2\le m\le(1-c)n/k$ and the third inequality follows from $\epsilon \ll c$.
This shows that $H(\mathcal{F})$ is not $\epsilon$-close to $H_S(n,m,k)$.
\end{proof}

To deal with the case when $H({\cal F})$ is not close to $H_D(n,m,3)$,
we first find a small matching $M$ in $H^*({\cal F})$ such that $M$ can
``absorb'' small vertex sets and $H^*({\cal F})-V(M)$ has an almost
perfect matching. When ${\cal F}$ is stable, the matching $M$ can
be found very easily by the following lemma and its
proof.

\begin{lem}\label{absorb}
Let $k$ be a fixed positive integer and let $0<\gamma'\ll\gamma \ll c \ll 1$ be reals.
Let $n, m$ be positive integers with $n/3k^2\le m\le (1-c)n/k$. Let $\mathcal{F}= \{{F}_1, ... ,{F}_m\}$ be a stable family of $k$-graphs such that $V(F_i)=[n]$ and $|{F}_i|> f(n,m,k)$ for all $i\in [m]$.
Then for sufficiently large $n$, $H^*(\mathcal{F})$ has a matching $M$ with $|M| \le \gamma n$ such that for any set $S\subseteq V(H^*(\mathcal{F}))\setminus V(M)$ with $|S|\le \gamma' n$ and $k|S\cap (\mathcal{V}\cup \mathcal{U})| = |S\cap [n]|$, $H^*(\mathcal{F})[V(M)\cup S]$ has a perfect matching.
\end{lem}
\begin{proof}
Recall that $\mathcal{V}= \{v_1,...,v_m\}$ and $\mathcal{U}= \{u_1,...,u_r\}$, where $r = \lfloor  n/k\rfloor-m$.
Fix an integer $t$ satisfying $\gamma' n  < t < \gamma n $. Then  $t
<\gamma n\leq \lfloor cn/k\rfloor\leq \lfloor n/k\rfloor-m=r$.
Let $s=\lceil n/3k^2\rceil-1$.

By Theorem~\ref{n>k^2} (viewing all $k$-graphs as the same $k$-graph),
since $|F_i|> f(n,m,k)\geq f(n,s,k)$ for all $i\in [m]$, every $F_i$ has a matching of size $s$.
Since $F_i$ is stable, $F_i[[s]]$ is a complete $k$-graph. Hence,
\begin{itemize}
\item[(i)] for any $i_1,i_2,..., i_k \le kt \le k\gamma n < s$ and
  $j\in [m]$, we have $\{v_j,i_1,i_2,...,i_k\} \in H^*(\mathcal{F})$.
\end{itemize}
From the definition of $H^*({\cal F})$, we have
\begin{itemize}
\item[(ii)] for any $i_1,i_2,...,i_k \in [n]$ and $j\in [r]$, $\{u_j,i_1,i_2,...,i_k\} \in H^*(\mathcal{F})$.
\end{itemize}	

Since $t<r$, we may choose a matching $M$ of size $t$ in $H^*(\mathcal{F})$ with $V(M)= \{u_1,..., u_t\}\cup [kt]$.
Note that $|M|=t\le \gamma n$.
We claim that this $M$ is the desired matching.
To see this, consider any subset $S$ with $S\cap V(M)=\emptyset$, $|S|\le \gamma' n$, and $k|S\cap (\mathcal{V}\cup \mathcal{U})| = |S\cap [n]|$.
Let $t' = |S\cap (\mathcal{V}\cup \mathcal{U})|$. So $t'\le \gamma'n < t$.
Then by (i) and (ii), there is a perfect matching $M_1$ in $H^*(\mathcal{F})[S\cap (\mathcal{V}\cup \mathcal{U})) \cup [kt'] ]$ .
By (ii), there exists a perfect matching $M_2$ in $H^*(\mathcal{F})[(V(M)\cup S)\setminus V(M_1)]$.
So $M_1\cup M_2$ is a perfect matching in $H^*(\mathcal{F})[V(M)\cup S]$.
\end{proof}

For the ``absorbing'' matching $M$ in $H^*({\cal F})$ in
Lemma~\ref{absorb}, we also want  $H^*({\cal F})-V(M)$ to have an almost
perfect matching. For this we need to use the following result of Frankl and R\"odl \cite{FR85}.
	
\begin{thm}[\cite{FR85}]\label{regular}
For every integer $k\ge 2$ and any real $\sigma >0$, there exist $\tau = \tau(k,\sigma)$ and $d_0 = d_0(k,\sigma)$ such that for every integer $n\ge D\ge d_0$ the following holds:
Every $n$-vertex $k$-graph $H$ with $(1-\tau)D<\Delta_1(H)< (1+\tau)D$ and $\Delta_2(H)<\tau D$ contains a matching covering all but at most $\sigma n$ vertices.
\end{thm}

In order to obtain a $k$-graph $H$ satisfying
Theorem~\ref{regular}, we use the approach from \cite{AFH12} by
conducting two rounds of randomization on $H^*(\mathcal{F})-V(M)$.
We summarize part of the proof in \cite{AFH12} (more precisely, the proof of Claim 4.1) as a lemma.
A \emph{fractional matching} in a $k$-graph $H$ is a function $w : E(H)\to [0, 1]$ such that for any
$v \in V(H)$, $\sum_{\{e\in E(H): v\in e\}} w(e) \le 1$.
A fractional matching is called \emph{perfect} if $\sum_{e\in E(H)} w(e) =|V(H)|/k$.

\begin{lem}[\cite{AFH12}, retained from the proof of Claim 4.1]\label{spaning graph}
Let $k\ge 3$ and $H$ be a $k$-graph on at most $2n$ vertices. 
Suppose that there are subsets $R^i\subseteq V(H)$ for $i=1,...,n^{1.1}$ satisfying the following:
\begin{itemize}
\item[(a).] every vertex $v\in V(H)$ satisfies that $|\{i: v\in R^i\}|=(1+o(1))n^{0.2}$,
\item[(b).] every pair $\{u,v\}\subseteq V(H)$ is contained in at most two sets $R^i$,
\item[(c).] every edge $e\in H$ is contained in at most one set $R^i$, and
\item[(d).] for every $i =1,...,n^{1.1}$, $R^i$ has a perfect fractional matching $ w^i$.
\end{itemize}
Then $H$ has a spanning subgraph $H'$ such that $d_{H'}(v) = (1+o(1))n^{0.2}$ for all $v\in V(H')$ and $\Delta_2(H')\le n^{0.1}$.
\end{lem}

We will also need to control the indepence number of random subgraphs
of  $H^*({\cal F})-V(M)$. The intuition is that when $H({\cal F})$ is
not close to $H_D(n,m,k)$ or $H_S(n,m,k)$,  $H^*({\cal F})-V(M)$ does not
have very large independence number. The following lemma in
\cite{LYY19} was proved by Lu, Yu, and Yuan using the container
method.

\begin{lem}[\cite{LYY19}, Lemma 5.4]\label{subgrpah indenpent set}
	Let $d,\epsilon',\alpha$ be positive reals and let  $k,n$ be positive integers. Let $H$ be an $n$-vertex $k$-graph such that $e(H) \ge dn^k$ and $e(H[S]) \ge \epsilon' e(H)$ for all $S\subseteq V(H)$ with $|S|> \alpha n$. Let $R\subseteq V(H)$ be obtained by taking each vertex of $H$ uniformly at random with probability $n^{-0.9}$. Then for any positive real $\gamma \ll \alpha$, the size of maximum independent sets in $H[R]$ is at most $(\alpha +\gamma)n^{0.1}$ with probability at least $1-(n^{O(1)}e^{-\Omega(n^{0.1})}) $
\end{lem}

We need an inequality on the function $f(n,m,k)$ proved by Frankl in
\cite{F17}.

\begin{lem}[\cite{F17}, Proposition 5.1]\label{lem:1}
Let $n,m,k$ be positive integers with $n\ge km-1$. Then $f(n,m,k) \ge f(n-1,m-1,k) + \binom{n-1}{k-1}$.
\end{lem}

We conclude this section with the well known Chernoff inequality.

\begin{lem}[Chernoff Inequality, see \cite{pr_book}]\label{chernoff}
	Suppose $X_1,...,X_n$ are independent random variables taking values in $\{0,1\}$. Let $X=  \sum_{i=1}^n X_i$ and $\mu = \mathbb{E}(X)$. Then  for any $0<\delta\le 1$,
	\begin{align}\label{neq:1}
	\mathbb{P}[X\ge(1+\delta)u]\le e^{-\delta^2u/3}
	\ and\  \mathbb{P}[X\le(1-\delta)u]\le e^{-\delta^2u/3}.
	\end{align}
	
\noindent In particular, if $X\sim \bin(n,p)$ and $\lambda < \frac{3}{2}np$, then
	\begin{align}\label{neq:2}
	\mathbb{P}([|X -np|]\ge  \lambda) \le  e^{-\Omega (\lambda^2/np)} .
	\end{align}
\end{lem}

\section{Extremal configuration $H_D(n,m,3)$}\label{Sec:HD}

From Lemma~\ref{claim:Delta} and Lemma~\ref{lemma:H_S(n,m)}, we see that if ${\cal F}$ is a
saturated family of $k$-graphs on vertex set $[n]$ and $H({\cal F})$ is
close to the extremal configuration $H_S(n,m,k)$ then there exist $F\in
{\cal F}$ and $v\in [n]$ such that $d_F(v)={n-1\choose k-1}$.
Such vertices $v$ can be removed from all $k$-graphs in ${\cal F}\setminus
\{F\}$ to obtain a smaller family ${\cal F}'$, so that if ${\cal F'}$
admits a rainbow matching then ${\cal F}$ admits a rainbow
matching.

In this section, we consider the case when $H({\cal F})$ is close to
$H_D(n,m,3)$ and $\mathcal{F}$ is stable.

\begin{lem}\label{lemma:H_D(n,m)}
Let $\epsilon, c$ be reals such that $0<\epsilon \ll c\ll 1$. Let $n,m$ be positive integers such that $ n/27\leq m\leq (1-c)n/3$.
Let $\mathcal{F}= \{F_1, ... ,F_m\}$ be a stable family of $3$-graphs on vertex set $[n]$ such that $|F_i|> f(n,m,3)$ for all $i\in [m]$.
If $H(\mathcal{F})$ is $\epsilon$-close to $H_D(n,m,3)$, then $\mathcal{F}$ admits a rainbow matching.
\end{lem}

\begin{proof}
Let $b = 6\epsilon^{1/6}n$.
If $F_i$ is $\sqrt{\epsilon}$-close to $D(n,m,3)$, then $F_i$ contains a
complete subgraph of size $3m-b$; for,
otherwise, as $F_i$ is stable, we have $|E(D(n,m,3))\setminus E(F_i)|\geq \binom{b}{3} > \sqrt{\epsilon}n^3$, a contradiction.

We claim that for any $i\in [m]$ and $j\in \{0,...,b\}$,  $\{2j+1,2j+2,3m-j\}\in {F}_i$.
To prove this claim we fix $i\in [m]$. Suppose for a contradiction that there exists an integer $t$ with $0\leq t\leq b$ such that $\{2t+1,2t+2,3m-t\}\notin F_i$.
Since $|F_i|>{3m-1\choose 3}$ and $F_i$ is stable, we have $\{1,2,3m\}\in F_i$.
So $t\geq 1$.
We now count the edges in $F_i$:
Let $q_1$ be the number of edges of $F_i$ in $[3m-1]$, and  $q_2$ be
the number of edges of $F_i$ not contained in $[3m-1]$.
Since $F_i$ is stable and $\{2t+1,2t+2,3m-t\}\notin F_i$,
we see that $\{a,b,c\}\notin F_i$ when $2t+2\leq a<b<3m-t\leq c\leq 3m-1$.
So $q_1\le \binom{3m-1}{3} -t\binom{3m-3t-3}{2}$.
Since $\{2t+1,2t+2,3m-t\} \notin F_i$, we have, for any $e\in F_i$ with $e\cap ([n]\setminus [3m-1])\ne \emptyset$, $e\cap [2t] \ne \emptyset$.
This shows $q_2 \le 2t(n-3m+1) n$.
First suppose that $n\le  7m/2$.
Then we have
\begin{align*}
|F_i|&\le \binom{3m-1}{3} -t\binom{3m-3t-3}{2}+ 2tn(n-3m+1)\\
&\le\binom{3m-1}{3} - t\left[\binom{3m-3t-3}{2} - 7m(m/2+1)\right]<\binom{3m-1}{3},
\end{align*}
where the second inequality holds since $n \le 7m/2$, and the last inequality holds since $t\le b=6\epsilon^{1/6}n \ll m $, a contradiction.
So we may assume $n > 7m/2$.
Let $m =\alpha n$; then $1/27\le \alpha < 2/7$.
We assert that $\binom{n}{3}- \binom{n-m+1}{3}>\binom{3m-1}{3} +2t n^2$.
To see this, let $f(x) = 1-(1-x)^3-(3x)^3$, so $$\frac{6}{n^3}\left(\binom{n}{3}- \binom{n-m+1}{3}- \binom{3m-1}{3}\right)=f(\alpha)+o(1).$$
Since $f'(x) = 3(1-2x-26x^2)$ is decreasing in $[1/27,2/7]$ with $f'(1/27)>0$ and $f'(2/7)<0$,
we have $f(\alpha)\geq \min\{f(1/27), f(2/7)\} = f(2/7)=\frac{2}{343} $ for $1/27\leq \alpha < 2/7$.
This shows that $\binom{n}{3}- \binom{n-m+1}{3}- \binom{3m-1}{3}=\frac{f(\alpha)}6n^3+o(n^3)\geq 2tn^2$, as asserted.
Then it follows that
\begin{align*}
|F_i|\le \binom{3m-1}{3} -t\binom{3m-3t-3}{2}+ 2tn(n-3m+1)< \binom{3m-1}{3} +2t n^2<\binom{n}{3}- \binom{n-m+1}{3},
\end{align*}
a contradiction as $|F_i| >f(n,m,3)$. This finishes the proof of Claim.
		
\medskip

Recall $\mathcal{V}= \{v_1,...,v_m\}$ from the definition of $H(\mathcal{F})$.
By the above claim, $M_1: = \{\{v_{i},2i-1,2i,3m-i+1\}: i\in [b]\}$ is a matching in $H(\mathcal{F})$.
Without loss of generality, let $F_{1},...,F_{a}$ be all $k$-graphs in $\mathcal{F}$ which are not $\sqrt{\epsilon}$-close to $D(n,m,3)$.
Since $H(\mathcal{F})$ is $\epsilon$-close to $H_D(n,m,3)$, we have $a\le\sqrt{\epsilon}n < b$.
Then for any $j\in [m]\backslash [b]$, since $F_j$ is $\sqrt{\epsilon}$-close to $D(n,m,3)$, $F_j$ contains a complete subgraph with size at least $3m-b$.
Hence we have $\{2j-1,2j,3m-j+1\} \in F_j$.
So $M_2: = \{\{v_j,2j-1,2j,3m-j+1\}: b <j\le m\}$ is a matching in $H(\mathcal{F})$ which is disjoint from $M_1$.
Then $M_1 \cup M_2$ forms a matching of size $m$ in $H(\mathcal{F})$.
So $\mathcal{F}$ admits a rainbow matching, completing the proof of Lemma~\ref{lemma:H_D(n,m)}.
\end{proof}

\section{A stability lemma}

In this section, we prove a result for stable
3-graphs, which may be viewed as a stability version of the following
result of {\L}uczak and Mieczkowska proved in  \cite{LM14}.

\begin{thm}[\cite{LM14}]\label{matching}
There exists positive integer $n_1$ such that for integers $m,n$ with $n\ge n_1$ and $1\le m\le n/3$, if $H$ is an $n$-vertex $3$-graph with $e(H) > f(n,m,3)$, then $\nu (H)\ge m$.
\end{thm}

Building on the proof in \cite{LM14},  we prove the following.

\begin{lem}\label{stability}
For any real $\epsilon>0$, there exists positive integer $n_1(\epsilon)$ such that the following holds. Let $m,n$ be integers with $n\geq n_1(\epsilon)$ and $1\le m\le n/3$, and let $H$ be a stable $3$-graph on the vertex set $[n]$. If $e(H)> f(n,m,3) -\epsilon^4 n^3$ and $ \nu(H) < m$,
then $H$ is $\epsilon$-close to $S(n,m,3)$ or $D(n,m,3)$.
\end{lem}
\begin{proof}
Suppose that $e(H)> f(n,m,3) -\epsilon^4 n^3$ and $ s:=\nu(H) < m$.
Let $M = \{(i_\ell,j_\ell,k_\ell): \ell\in [s]\}$ be a largest matching in $H$ and partition $V(M) = I\cup J\cup K$ such that every edge $(i,j,k)\in E(M)$ with $i<j<k$ satisfies $i\in I,j\in J$ and $k\in K.$
Since $H$ is stable, we may choose $V(M)$ to be $[3s]$.

Let $V'=[n]\backslash [3s]$. For $x\in [3s]$, let $e(x)$ denote the
edge in $M$  containing $x$. Let  $F_1 = \{\{v\}\in \binom{[3s]}{1}: d_{V'}(v)\ge
20n\}$, $F_2 = \{\{v,w\}\in \binom{[3s]}{2}: e(v)\ne e(w)$ and
$d_{V'}(v,w)\ge 20\}$, and $F_3 =\{ \{u,v,w\}\in \binom{[3s]}{3}:
e(u)$, $e(v)$ and $e(w)$ are pairwise distinct$\}$.
Let $H^* = ([3s],F)$ be the hypergraph with vertex set $[3s]$ and edge set $F=M \cup F_1\cup F_2 \cup F_3$.

Call an edge $e \in H$ \emph{traceable} if $e\cap [3s] \in
F$, and \emph{untraceable} otherwise. Since $M$ is a maximum matching in $H$,
$V'$ is independent in $H$. So
the number of untraceable edges of $H$ is bounded from above by
$$\binom{3s}{1}\cdot 20n+ \left(\binom{s}{2}\binom{3}{1}\binom{3}{1}\times 19 + \binom{s}{1}\binom{3}{2}n\right)+ \binom{s}{1}\binom{3}{2}\binom{3s-3}{1}\le 32n^2=o(n^3),$$
where we use $s<m\leq n/3$.
We point out that those edges (there being  $o(n^3)$ of them)  will be negligible in the following proof.
	
Let $T$ be a triple of edges from $M$. We say $T$ is \emph{bad} if $V (T)$ contains three pairwise disjoint edges of $H^*$ whose union intersects $I$ in at most $2$ vertices, and \emph{good} otherwise.
For each $i\in [3]$, let $f_i(T)$ denote the number of edges of $F_i$ contained in $V (T)$. Note that $f_3(T)\leq 27$.
The following two claims are explicit in \cite{LM14}.
\begin{claim}
There exist no three pairwise disjoint bad triples (of edges in $M$).
Hence, there exist at most six edges in $M$ such that each bad triple contains one of these edges.
\end{claim}
	
	\begin{claim}\label{claim 3}
		Let T be a good triple.\\
		(i) If $f_3(T) \ge 24$, then $f_1(T) = f_2(T)=0$.\\
		(ii) If $f_3(T) = 20$, then $f_1(T) \le 1$ and $f_2(T)\le 12$.\\
		(iii) If $f_3(T) \le 19$, then $f_1(T)\le 3$ and
                $f_2(T) \le 15$. Moreover, the only triples $T$ for
                which $f_3(T) =19$, $f_2(T) =15$, and $f_1(T)=3$, are those in which each edge of $H^*$ contained in $V(T)$ intersects I.\\
		(iv) If $f_3(T) = 21$, then $f_1(T) \le 1$ and $f_2(T)\le 10$\\
		(v) If $22\le f_3(T) \le 23 $, then $f_1(T) = 0$ and $f_2(T)\le 7$
	\end{claim}
	
We remove exactly six edges from $M$ such that the resulting matching $M'$ only contains good triples.
Since $H$ has at most $18n^2$ edges intersecting $V(M\backslash M')$ and $32n^2$ untraceable edges, we have
$$e(H) \le |F_1|\binom{n-3s}{2} + |F_2| (n-3s) +|F_3|+ 50n^2.$$
To bound $|F_i|$, let us consider the summation of $f_i(T)$ over all $T \in \binom{M'}{3}$.
Since each edge from $F_i$ is counted exactly $\binom{(s-6)-i}{3-i}$ times in this sum, we have $|F_i|\binom{(s-6)-i}{3-i}=\sum_{T\in \binom{M'}{3}} f_i(T)$.
Therefore,
\begin{align*}
e(H) &\le  \sum_{T\in \binom{M'}{3}}\left(f_1(T)\frac{\binom{n-3s}{2}}{\binom{s-7}{2}} +f_2(T) \frac{n-3s}{s-8} +f_3(T) \right) + 50n^2\\
&\le \sum_{T\in \binom{M'}{3}}\left(f_1(T)\frac{(n-3s)^2}{s^2} +f_2(T) \frac{n-3s}{s} +f_3(T) \right) + O(n^2).
\end{align*}
Here, the last inequality is trivial when $s\leq 15$, and it holds when $s>15$ because the difference between the above two summations is at most
\begin{align*}
\sum_{T\in \binom{M'}{3}}\left( f_1(T)\frac{15(n-3s)^2}{s(s^2-15s)} +f_2(T)\frac{8(n-3s)}{s(s-8)}  \right) \le  \binom{s-6}{3} \left(\frac{45(n-3s)^2}{s(s^2-15s)} +\frac{120(n-3s)}{s(s-8)}\right)= O(n^2),
\end{align*}
where $3s<n$, $f_1(T)\le3$, and $f_2(T) \le 15$ (from Claim~\ref{claim 3}).
	
To further bound $e(H)$, we partition good triples $T$ depending on $f_3(T)$ and $f_1(T)$.
Let $T_i = \{T\in \binom{M'}{3}: f_3(T) =i\}$ for $i\in [27]$ and $X = \{T\in \binom{M'}{3}: f_1(T) = 3\}$.
Consider any $T\in X$; so $T$ is a good triple.\footnote{Since $T$ is good, the union of any three disjoint edges of $H^*$ in $V(T)$ must contain the three vertices in $V(T)\cap I$.}
Since $f_1(T) = 3$, the three edges of $F_1$ contained in $V(T)$ are precisely the three vertices in $V(T)\cap I$,
and each edge of $H^*$ contained in $V(T)$ intersects $I$.
Since $H$ is stable and $V(M)=[3s]$, using the definition of $F_1$, it is not hard to see that $X\subseteq T_{19}$.
    	
Define $x_1 = \sum_{i=1}^{18}|T_i| + |T_{19} \setminus X|, x_2 =|T_{20}|, x_3=|T_{21}|, x_4 = |T_{22}|+|T_{23}|, x_5 = \sum_{i=24}^{26}|T_i|, x=|X|$, and $y= |T_{27}|$.
So $\sum_{i=1}^{5} x_i +x +y=\binom{s-6}{3}$.
From now on, we let $t=(n-3s)/s$.
By Claim \ref{claim 3} and the fact $X\subseteq  T_{19}$, we can derive from the above upper bound on $e(H)$ that
\begin{align*}
e(H)&\le (3x+2x_1+x_2+x_3)t^2+(15x+15x_1+12x_2+10x_3+7x_4)t\\
&+(19x+19x_1+20x_2+21x_3+23x_4+26x_5+27y)+O(n^2).
\end{align*}
For convenience, we write
\begin{align*}
& f_t(x_1,x_2,x_3,x_4,x_5,x,y) = \sum_{i=1}^{5}\alpha_i(t)\cdot x_i + \beta_1(t)\cdot x + \beta_2(t)\cdot y,\\
\mbox{ where ~~~} &\alpha_1(t) = 2t^2 +15t +19, ~~~~ \alpha_2(t) = t^2 +12t +20, ~~~~ \alpha_3(t) = t^2 +10t +21\\
&\alpha_4(t)= 7t +23, ~~~~ \alpha_5(t) = 26, ~~~~ \beta_1(t) = 3t^2 +15t +19, \mbox{ ~~and ~~} \beta_2(t) = 27.
\end{align*}
Then it follows that
$$e(H)\leq f_t(x_1,x_2,x_3,x_4,x_5,x,y)+O(n^2).$$

Next, we derive properties of the functions $\alpha_i(t)$ and $\beta_j(t)$.
\begin{claim}\label{claim_cal}
For any $t\geq 0$, $\max \{\beta_1(t),\beta_2(t)\} \geq \max \{\alpha_1(t),\alpha_2(t),\alpha_3(t),\alpha_4(t),\alpha_5(t)\} +0.2$.
\end{claim}
\pf
We have $\beta_2(t) = 27$.
It is easy to see that for each $i\in [5]$, the functions $\alpha_i(t),\beta_1(t) -\alpha_i(t)$ and $\beta_1(t)$ are increasing for $t\geq 0$.
Note that $\beta_1(0.5) = 27.25$, $\alpha_2(0.5) = 26.25$, $\alpha_3(0.5) = 26.25$ and $\alpha_4(0.5) = 26.5$;
so $\max\{\beta_1(t),27\}\ge \alpha_i(t) + 0.2$ for $t\geq 0$ and $i=2,3,4$.
Since $\beta_1(t) -\alpha_1(t) = t^2$, and $\alpha_1(\sqrt{0.2}) < 27-0.2$, we see $\max\{\beta_1(t),27\}\ge \alpha_1(t) + 0.2$ for all $t\geq 0$.
\qed
	
\medskip

Since $\beta_1(t)\binom{s-6}{3} \le \frac12(n-3s)^2 s + \frac52(n-3s)s^2+\frac{19}{6}s^3=\frac16n^3-\frac16(n-s)^3$,
we see $\max\{\beta_1(t), \beta_2(t)\} \binom{s-6}{3}\le \max\left\{\binom{n}{3} - \binom{n-s+1}{3},\binom{3s-1}{3}\right\} + O(n^2)=f(n,s,3)+ O(n^2).$	
By Claim~\ref{claim_cal} and the fact that $\sum_{i=1}^{5} x_i +x +y=\binom{s-6}{3}$, we have
\begin{align}\label{equ:f_t}
\begin{split}
&f_{t}(x_1,x_2,x_3,x_4,x_5,x,y)\le \bigg(\max\{\beta_1(t), \beta_2(t)\}-0.2\bigg)\sum_{i=1}^5 x_i + \beta_1(t)x + \beta_2(t)y\\
\le & \max\{\beta_1(t), \beta_2(t)\}\binom{s-6}{3}-0.2\sum_{i=1}^5x_i\leq f(n,s,3)-0.2\sum_{i=1}^5x_i+O(n^2).
\end{split}
\end{align}	

Let $\cup X$ (respectively, $\cup T_{27}$) denote the set of edges each of which belongs to some triple in $X$ (respectively, in $T_{27}$).
Now we show the following claim.
\begin{claim}\label{x,y}
$s> m- \epsilon n/4 $, and $x >\binom{s-6}{3} - 10\epsilon^4 n^3- \binom{\epsilon n/24}{3} $ or $y > \binom{s-6}{3} - 10\epsilon^4 n^3- \binom{\epsilon n/12}{3}$.
\end{claim}
\pf
If $s\le m-\epsilon n/4$, then by \eqref{equ:f_t} we have
\begin{align*}
e(H)&\leq f_{t}(x_1,x_2,x_3,x_4,x_5,x,y)+ O(n^2)\le f(n,s,3) + O(n^2)\\
&\le f(n,m,3) -\binom{\epsilon/4n}{3} +O(n^2)\le f(n,m,3) - \epsilon^4 n^3,
\end{align*}
a contradiction. So $s> m- \epsilon n/4 $.
First we see that $x+y> \binom{s-6}{3} - 10\epsilon^4 n^3$; for, otherwise, $\sum_{i=1}^5 x_i\geq 10\epsilon^4 n^3$, which together with \eqref{equ:f_t} implies
$$e(H)\le f_{t}(x_1,x_2,x_3,x_4,x_5,x,y)+O(n^2)\le f(n,m,3) - 2\epsilon^4n^3  + O(n^2)\le f(n,m,3) -\epsilon^4 n^3,$$ a contradiction.
Now suppose that $x > \binom{\epsilon n/12}{3}$ and $y > \binom{\epsilon n/24}{3}$.
Then $|\cup X|>\epsilon n/12$ and $|\cup T_{27}| > \epsilon n/24$.
For any edge $e = (i,j,k) \in \cup X$ with $i<j<k$, by the previous discussion, we have $i\in F_1$.
For any edge $e = (i,j,k) \in\cup T_{27}$ with $i<j<k$, by Claim~\ref{claim 3} we see $i \notin F_1$.
Thus $(\cup X) \cap (\cup T_{27}) = \emptyset$.
The triples $T = \{e_1,e_2,e_3\}$ with $e_1 \in \cup X$ and $e_2,e_3\in \cup T_{27}$ cannot satisfy both $f_3(T) =27$ and $f_1(T) =3$.
This shows $x+y <\binom{s-6}{3} -|\cup X|\binom{| \cup T_{27}|}{2} \le\binom{s-6}{3} -  \frac{\epsilon n}{12}\binom{\epsilon n/24}{2} $,
contradicting that $x+y > \binom{s-6}{3} - 10\epsilon^4 n^3$.
Hence, we have that either $x \le \binom{\epsilon n/12}{3}$ or $y \le \binom{\epsilon n/24}{3}$.\qed

\medskip

Suppose $x > \binom{s-6}{3} - 10\epsilon^4 n^3- \binom{\epsilon n/24}{3}$.
So $x > \binom{s-6}{3} - \binom{\epsilon n/12}{3}$ and thus $|\cup X| > s-6 - \epsilon n/12$.
Recall that for any $T\in X$, $T$ is a good triple and, hence, each edge of $H^*$ contained in $V(T)$ intersects $I$.
Hence any traceable edge which intersects $V(\cup X)$ must also intersect $I$.
Thus, the number of edges of $H$ not intersecting $I$ is at most $|V(M')\setminus V(\cup X)|\binom{n}{2}+ 50n^2\leq \frac{\epsilon n}{4}\binom{n}{2}+50n^2\leq \frac{\epsilon}{4}n^3$.
As $|I| = s\le m-1$, $$|E(S(n,m,3))\backslash E(H)| = |E(H)\backslash E(S(n,m,3))| +e(S(n,m,3))-e(H)\le \frac{\epsilon}{4}n^3 +\epsilon^4 n^3 < \epsilon n^3.$$
So in this case we see that $H$ is $\epsilon$-close to $S(n,m,3)$.
	
By Claim~\ref{x,y}, it remains to consider $y >\binom{s-6}{3} - 10\epsilon^4 n^3- \binom{\epsilon n/12}{3} $.
We claim that there exists a complete $3$-graph $K$ on more than $3m-3\epsilon n/2$ vertices and $V(K)\subseteq V(M')$.
Suppose to contrary that $V(M')$ does not contain such a complete $3$-graph $K$.
Since $|V(M')| - (3m-3\epsilon n/2)=3(s-6)-3m+3\epsilon n/2 > \frac{\epsilon n}{2}$ and $H$ is stable,
$V(M')$  contains an independent set of size $\frac{\epsilon n}{2}$, say $A$.
Note that if $T =\{e_1,e_2,e_3\}$ with $e_i \cap A \ne \emptyset$ for all $i\in [3]$, then $f_3(T) < 27$.
Since there are at least $|A|/3 \ge \epsilon n/6$ edges in $M'$ which intersect with $A$,
we see that $y \le \binom{s-6}{3} - \binom{\epsilon n/6}{3}$, a contradiction.

Then $|E(D(n,m,3)\backslash E(H)|\le|E(D(n,m,3)\backslash E(K)| \le \frac32\epsilon n\binom{n}{2}< \epsilon n^3$,
i.e., $H$ is $\epsilon$-close to $D(n,m,3)$.
This finishes the proof of Lemma~\ref{stability}.	
\end{proof}	

\section{Almost perfect rainbow matchings}

In this section, we prove a lemma about almost perfect rainbow
matchings that we will need. In fact, this result holds for families of
$k$-graphs, for any $k\ge 3$.

\begin{lem}\label{large-mat}
For any given integer $k\geq 3$, there exist positive reals $c$ and $n_2$ such that the following holds.
Let $n,m $ be integers with $n\geq km$ and $n\geq n_2$,
and let $\mathcal{F}=\{F_1,...,F_m\}$ be a stable family of $k$-graphs on the same vertex set $[n]$ such that $|F_i|> \binom{km-1}{k}$ for each $i\in [m]$.
If $m>(1-c)n/k$, then $\mathcal{F}$ admits a rainbow matching.
\end{lem}
\begin{proof}
We choose $c' = c'(k)$ and $c = c(k)$ small enough such that $0<c\ll c' \ll 1$.
Let $n$ be sufficiently large and $n/k\geq m>(1-c)n/k$.
Suppose to the contrary that $|F_i|> \binom{km-1}{k}$ for each $i\in [m]$ and $\mathcal{F}$ does not admit a rainbow matching.

By Corollary~\ref{coro:stable-saturated}, we may additionally assume
$\mathcal{F}$ is saturated.
Let $U_i$ be the vertex set of a largest complete $k$-graph in $F_i$ for $i\in [m]$.
Since $F_i$ is stable, we may choose $U_i = [|U_i|]$ such that $[n]\setminus U_i$ is an independent set in $F_i$.
For each $i\in [m]$, we have $|U_i|>(1-c')km$, for, otherwise, we have the following contradiction for some $i \in [m]$:
\begin{align*}
|F_i|\le \binom{n}{k} - \binom{c'km}{k}\le \binom{n}{k} - (cn+1)\binom{n-1}{k-1}\le \binom{n}{k} - (n-km+1)\binom{n-1}{k-1}<\binom{km-1}{k},
\end{align*}
where the second inequality holds since $c\ll c'\ll 1$ and $m>(1-c)n/k$, the third
inequality holds since $n-km<cn$,
and the last inequality holds since $\binom{n}{k} -\binom{km-1}{k} = \sum_{i=1}^{n-km+1} \binom{n-i}{k-1} < (n-km+1)\binom{n-1}{k-1}$.

Let $U = \bigcap_{i=1}^m U_i$. By the above paragraph, we see that $|U|\ge (1-c')km$.
If $|U|\ge km$, then it is clear that $\mathcal{F}$ admits a rainbow matching.
So we may assume that $U_m= U\subseteq [km-1]$.
Because $U_m$ is the vertex set of a largest complete $k$-subgraph of
$F_m$ and since $F_m$ is stable and $|F_m|>{km-1\choose k}$,
there exists some $k$-set $e\notin F_m$ such that $|e\cap U|=k-1$ and $km \in e$.
Since $\mathcal{F}$ is saturated, there exists a rainbow matching $M$ in $\mathcal{F}\setminus F_m$ such that
$M\cup \{e\}$ is a rainbow matching in ${\cal F}(e,F_m)$.
Since $F_i$ is stable for each $i\in [m]$, we may assume that $V(M)\cup e= [km]$.
Let $M' = \{e'\in M: e' \not\subseteq U\}$. 

\medskip
\noindent\textbf{Claim.}
	(a) $|M'| < c'km$,\\
	(b) Each edge of $F_m$ is contained in $U$ or intersects an edge of $M'$, and \\
	(c) For any $v\in V(M)\setminus U$, $d_{F_m[U]}(v)\leq  c'k^2m\binom{|U|}{k-2}$.

\medskip

\pf
To prove (a), just observe that $|M'|\le |V(M)\setminus U| = (km-1)-|U|< c'km$.
	
Suppose (b) fails. That is, there exists an edge $f\in F_m$ such that
$f\backslash U\neq \emptyset$ and $f\cap V(M')=\emptyset$. Note  that
$f\cap (U\backslash V(M'))\neq \emptyset$, as $[n]\setminus U$ is
independent in $F_m$.
In particular, $|f\cap (U\backslash V(M'))|\le k-1$.
Let $|M'|=m-t$ for some $t\geq 1$. Recall that $U\cup V(M') = V(M) = [km-1]$. Hence $|U\backslash V(M')|=kt-1$
and, thus, $U\backslash (V(M')\cup f)$ induces a common complete $k$-graph of size at least $k(t-1)$ in all $F_i$.
Then we see that $M'\cup \{f\}$ together with a matching of size $t-1$ in $U\backslash (V(M')\cup f)$ form a rainbow matching for $\mathcal{F}$.
So $(b)$ holds.
	
Now we prove (c). For any $v\in V(M)\setminus U\subseteq [km]$, by the maximality of $U$,
there exists $f\in {[n]\choose k}\setminus F_m$ such that $v\in f$ and $|f\cap U| = k-1 $.
So there exists a rainbow matching $N$ in $\mathcal{F}\setminus F_m$
such that $N\cup \{f\}$ is a rainbow matching in ${\cal F}'(f,F_m)$.
Since $F_i$ is stable for $i\in [m]$, we may assume that $V(N)\cup f = [km]$.
Let $N' = \{e'\in N: e'\not\subseteq U\}$. By applying (b) to $N'$,
every edge of $F_m$ containing $v$ intersects $V(N')$.
Since $V(N') \le k|N'|\le k(km-|U|)\le c'k^2m$, there are at most $c'k^2m\binom{|U|}{k-2}$ edges $e'$ in $F_m$ containing $v$ such that $e' \subseteq U\cup \{v\}$.
Hence (c) holds. This proves the claim.\qed
		
\medskip

Note that $|e\cap U|=k-1$ and $V(M) \cup U= [km-1]$.
Let $q_1$ be the number of edges of $F_m$ contained in $[km-1]$, and
$q_2$ be the number of edges of $F_m$ with at least one vertex in $[n]\setminus [km-1]$.
By (c), we have $$q_1\le \binom{km-1}{k} - |V(M)\setminus U|\binom{|U|}{k-1} + |V(M)\setminus U|\cdot c'k^2m\binom{|U|}{k-2}.$$
By (b), we see $q_2 \le  \left|V(M')\right|\cdot (n-km+1)\binom{n-2}{k-2}$.
So  we have
\begin{align*}
|F_m| &\le \binom{km-1}{k} - |V(M)\setminus U| \left[\binom{|U|}{k-1} + c'k^2m\binom{|U|}{k-2}\right] + \left|V(M')\right| (n-km+1)\binom{n-2}{k-2}\\
&\le \binom{km-1}{k} -|V(M)\setminus U| \left[\binom{|U|}{k-1} +c'k^2m\binom{|U|}{k-2}\right]+k|V(M)\setminus U|(cn+1)\binom{n-2}{k-2}\\
&= \binom{km-1}{k} - |V(M)\setminus U|\cdot \left[\binom{|U|}{k-1}
  -c'k^2m\binom{|U|}{k-2}-k(cn+1)\binom{n-2}{k-2} \right]\\
& <\binom{km-1}{k},
\end{align*}
where the second inequality holds since $n-km<cn$ and $|M'|\le |V(M)\setminus U|$, and
the last inequality holds since $c',c$ are small enough and $|U| >(1-c')km>(1-c')(1-c)n$.
This is a contradiction, finishing the proof of Lemma~\ref{large-mat}.
\end{proof}

\section{Non-extremal configurations}

Note that if there exist $F\in
{\cal F}$ and $v\in [n]$ such that $d_F(v)={n-1\choose k-1}$ then $v$
can be removed from all $k$-graphs in ${\cal F}\setminus
\{F\}$ to obtain a smaller family ${\cal F}'$ so that ${\cal F'}$
admits a rainbow matching if and only if ${\cal F}$ admits a rainbow
matching. Hence, if such vertex does not exist in a saturated family
${\cal F}$, then
from  Lemma~\ref{claim:Delta}, we see that  $d_{F}(v)\le
\binom{n-1}{k-1} - \binom{n-k(m-1)-1}{k-1}$ for all $v\in F$ and $F\in
{\cal F}$. This leads us to the following result.

\begin{lem}\label{lem:not close}
Given reals $0<\epsilon \ll c\ll 1$, let $n\geq n(\epsilon, c)$ be a sufficiently large integer and $m$ be an integer such that $n/27<m<(1-c)n/3$.
Let $\mathcal{F}= \{F_1, ... ,F_m\}$ be a stable family of $3$-graphs on vertex set $[n]$ such that for every $i\in [m]$, $|F_i|> f(n,m,3)$ and $d_{F_i}(v)\le  \binom{n-1}{2} - \binom{n-3(m-1)-1}{2}$  for each $v\in [n]$.
If $H(\mathcal{F})$ is $\epsilon$-close to neither $H_S(n,m,3)$ nor $H_D(n,m,3)$, then $\mathcal{F}$ admits a rainbow matching.
\end{lem}

\begin{proof} Given $0<\epsilon \ll c\ll 1$, let $n', m'$ be integers such that $n'$ is sufficiently large and $n'/27<m'<(1-c)n'/3$.
Let $\mathcal{F}= \{F_1, ... ,F_{m'}\}$ be a family of $3$-graphs on the
vertex set $[n']$ such that $|F_i|>f(n',m',3)$ and $d_{F_i}(v)\le
\binom{n'-1}{2} - \binom{n'-1-3(m'-1)}{2}$ for $i\in [m']$ and
$v\in [n']$.
Suppose that $H(\mathcal{F})$ is not $\epsilon$-close to $H_S(n',m',3)$ or $H_D(n',m',3)$.
Our ultimate goal is to find a rainbow matching in $\mathcal{F}$.

Let $n' =3m'+3r'+s$ where $0 \le s <3$.
Recall the definitions of $H(\mathcal{F})$ and $H^{*}(\mathcal{F})$ such that $V(H(\mathcal{F}))=[n']\cup \mathcal{V}'$ and $V(H^*(\mathcal{F}))=[n']\cup \mathcal{V}'\cup \mathcal{U}'$, where $|\mathcal{V}'|=m'$ and $|\mathcal{U}'|=r'$.
By Lemma~\ref{absorb}, for $ 0<\gamma'\ll \gamma \ll \epsilon \ll c \ll1 $, there exists a matching $M_a$ in $H^*(\mathcal{F})$ with $|M_a| \le \gamma n'$
such that for any $S\subseteq  V(H^{*}(\mathcal{F}))\backslash V(M_a)$ with $|S|\le \gamma' n'$ and $3|S\cap (\mathcal{V}'\cup \mathcal{U}')| = |S\cap [n']|$, $H^{*}(\mathcal{F})[V(M_a)\cup S]$ has a perfect matching.
In the rest of the proof, without loss of generality, we use the following notation:
$$H =H^{*}(\mathcal{F})-V(M_a), [n] = [n']\setminus V(M_a),
\mathcal{V}=\mathcal{V}'\setminus V(M_a)= \{v_1,...,v_m\}, \mathcal{U}=\mathcal{U}'\setminus V(M_a)=\{u_1,...,u_r\}.$$

Then $n= 3m+3r+s$.
Using the above property of the matching $M_a$, it now suffices for us
to find an almost perfect matching in $H$.
To find this almost perfect matching, our plan is to show that there exists an almost regular subgraph of $H$ with bounded maximum co-degree so that Theorem~\ref{regular} can be applied.
To that end, in what follows we will use the two-round randomization technique developed in \cite{AFH12}.

Let $R$ be chosen from $V(H)$ by taking each vertex independently with probability $n^{-0.9}$.
We take $n^{1.1}$ independent copies of $R$ and denote them by $R^i$ for $1\le i\le n^{1.1}$.
For $S\subseteq V(H)$, denote $Y_S = |\{ i:S\subseteq R^i\}|$.
First we have the following claim.		

\medskip

\noindent\textbf{Claim A.}
With probability $1-o(1)$, the following hold:\\
(i) for every $v\in V(H)$, $Y_{\{v\}} = (1+o(1)) n^{0.2}$,\\
(ii) every pair $\{u,v\}\subseteq V(H)$ is contained in at most two sets $R^i$, and\\
(iii) every edge $e\in H$ is contained in at most one set $R^i$.

\medskip
\pf
Note that $Y_S\sim \bin(n^{1.1},n^{-0.9|S|})$ for any $S\subseteq V(H)$.
Thus, $\mathbb{E}[Y_{\{v\}}]= n^{0.2}$ for every $v\in V(H)$. By Lemma~\ref{chernoff}~\eqref{neq:2}, we have
$P(|Y_{\{v\}}-n^{0.2}|>n^{0.15}) \le e^{-\Omega(n^{0.1})}$. By union bound, we see $(i)$ holds.
To prove $(ii)$ and $(iii)$, let
$$Z_2 = \left| \left\{ \{u,v\} \in \binom{V(H)}{2}: Y_{\{u,v\}}\ge 3  \right\} \right| \mbox{ and } Z_3 = \left| \left\{ S \in \binom{V(H)}{3}: Y_S\ge 2  \right\} \right|.$$
Then $\mathbb{E}[Z_2]=\binom{|V(H)|}{2}P(Y_{\{u,v\}}\geq 3)\leq \binom{n}{2}(n^{1.1})^{3}(n^{-1.8})^3\leq 4n^{-0.1}$ and $\mathbb{E}[Z_3]\leq \binom{n}{3}(n^{1.1})^2(n^{-2.7})^{2}\leq 8n^{-0.2}$.
By Markov's inequality, we have
\begin{align*}
\mathbb{P}(Z_2  = 0) >1-4n^{-0.1} \mbox{ ~~and~~ } \mathbb{P}(Z_3= 0) >1-8n^{-0.2}.
\end{align*}
That implies that $(ii)$ and $(iii)$ hold with probability at least $1-4n^{-0.1}$ and $1-8n^{-0.2}$, respectively.
\qed
	
\medskip

Next we want to prove that there exists a  perfect  (or, rather,
maximum) fractional matching in each $H[R^i]$.
To do so, we define a maximal subset $R'^i\subseteq R^i$ that satisfies $R'^i\cap [n] = 3\lvert R'^i\cap (\mathcal{V}\cup \mathcal{U})) \rvert$ as follows.
If $\lvert R^i\cap [n]\rvert\geq 3\lvert R^i\cap (\mathcal{V}\cup \mathcal{U})) \rvert $, we take a subset of $R^i$ denote by $R'^i$, which is chosen from $R^i$ by deleting $\lvert R^i\cap [n]\rvert - 3\lvert R^i\cap (\mathcal{V}\cup \mathcal{U}))\rvert $ vertices in $R^i\cap [n]$ independently and uniformly at random.
Otherwise $\lvert R^i\cap [n]\rvert  < 3\lvert R^i\cap (\mathcal{V}\cup \mathcal{U}))\rvert$,
we take a subset of $R^i$ denote by $R'^i$ by the following two step:
First we delete at most $3$ vertices (chosen independently and uniformly at random) in $R^i\cap [n]$ so that the number $\ell$ of the remaining vertices is a multiple of 3;
then we delete $\lvert R^i\cap (\mathcal{V}\cup \mathcal{U}))\rvert -\ell/3$ vertices in $R^i\cap (\mathcal{V}\cup \mathcal{U}))$ independently and uniformly at random.

For $S\subseteq V(H)$,
define $Y'_S = \lvert \{ i:S\subseteq R'^i\}\rvert$.
Note that $\mathbb{E}(R^i\cap [n]) = n^{0.1}$, $\mathbb{E}(R^i\cap (\mathcal{V}\cup \mathcal{U})) = n^{0.1}/3$, and $\mathbb{E}(R^i\cap \mathcal{V}) = n^{-0.9}m$.
For each $i$, let $A_i$ be the event $ \big| |R^i\cap [n]|  -n^{0.1} \big| < n^{0.095} $, $B_i$ be the event $  \big| |R^i\cap (\mathcal{V}\cup \mathcal{U}))|  -n^{0.1}/3 \big| < n^{0.095}$, and $C_i$ be the event $\big| |R^i\cap \mathcal{V}|  -n^{-0.9}m \big| < n^{0.095}$.

\medskip

\noindent\textbf{Claim B.}
With probability $1-o(1)$, the following hold:\\
(i) $\bigwedge_i (A_i\wedge B_i \wedge C_i)$ holds,\\
(ii) for every $v\in V(H)$, $Y'_{\{v\}} = (1+o(1))n^{0.2}$,\\
(iii) every pair $\{u,v\}\subseteq V(H)$ is contained in at most two sets $R'^i$, and\\
(iv) every edge $e\in H$ is contained in at most one set $R'^i$.
\medskip

\pf
Since $R'^i \subseteq R^i$, it is clear from Claim A that (iii) and (iv) hold with probability $1-o(1)$.
Next we consider (i).
By Lemma~\ref{chernoff}~\eqref{neq:2} (with $\lambda= n^{0.095} $), for each $1\le i \le n^{1.1}$, we have
\begin{align*}
\mathbb{P}(\overline{A_i} ) \le e^{-\Omega (n^{0.09})}~, ~ \mathbb{P}(\overline{B_i} ) \le e^{-\Omega (3n^{0.09})} =e^{-\Omega (n^{0.09})}
\mbox{~ and ~} \mathbb{P}(\overline{C_i} ) \le e^{-\Omega (\frac{n}{m}n^{0.09} )} =e^{-\Omega (n^{0.09})}.
\end{align*}
Thus by union bound, $\mathbb{P}(\bigwedge_i (A_i\wedge B_i \wedge C_i)) = 1- o(1)$, proving $(i)$.
	
Assuming $A_i\wedge B_i \wedge C_i$, we see $|R^i\setminus R'^i|< \max\{\lvert R^i\cap [n]\rvert - 3\lvert R^i\cap (\mathcal{V}\cup \mathcal{U}))\rvert,\lvert R^i\cap (\mathcal{V}\cup \mathcal{U}))\rvert - \lfloor\lvert R^i\cap [n]\rvert/3 \rfloor +3 \} < 4n^{0.095}$.
Then by the choice of $R'^i$, for all $v\in V(H)$,
the probability $\mathbb{P}(\{v\in R^i\setminus R'^i \big| (A_i\land B_i\land C_i)\land (v\in R^i)\})$ is at most
\begin{align*}
\max \left\{\frac{|R^i\setminus R'^i|}{|R^i\cap[n]|},\frac{|R^i\setminus R'^i|}{|R^i\cap(\mathcal{V}\cup \mathcal{U}))|}\right\}\le \frac{|R^i\setminus R'^i|}{|R^i\cap(\mathcal{V}\cup \mathcal{U}))|}  <\frac{4n^{0.095}}{n^{0.1}/3-n^{0.095}} < 13 n^{-0.005}.
\end{align*}
Using coupling and applying Lemma~\ref{chernoff}~\eqref{neq:2} to $\bin(|Y_v|,13n^{-0.005})$ with $\lambda=3n^{0.195}$, we have
\begin{align*}
\mathbb{P}\left(\left\{Y_{\{v\}} - Y'_{\{v\}} > 16n^{0.195} \bigg| \bigwedge_i (A_i\wedge B_i \wedge C_i) \wedge \big(Y_{\{v\}}  = (1+o(1)) n^{0.2}\big) \right\}\right) \le e^{- \Omega(n^{0.195})}.
\end{align*}
Note that with probability $1-o(1)$, $\bigwedge_i (A_i\wedge B_i \wedge C_i)$ and $Y_{\{v\}}  = (1+o(1)) n^{0.2}$ hold for all $v\in V(H)$.
By union bound, we can derive that $0\leq Y_{\{v\}} - Y'_{\{v\}} \le 16n^{0.195} = o(n^{0.2})$ for all $v\in V(H)$ with probability $1-o(1)$.
Hence (ii) holds with probability $1-o(1)$. This proves Claim B.\qed

\medskip
	
Let $n_i = \lvert R'^i\cap [n]\rvert$ and $m_i = \lvert R'^i\cap \mathcal{V}\rvert$.
Using (i) of Claim B, we see that with probability $1-o(1)$, $m_i= (1+o(1))mn^{-0.9} =\Theta(n^{0.1})= \Theta(n_i)$ for all $1\leq i\leq n^{1.1}$.

\medskip

\noindent\textbf{Claim C.} 	
With probability $1-o(1)$, the following hold for all $1\le i\le n^{1.1}$:\\
(a) $H[R'^i \setminus \mathcal{U}]$ is not $\epsilon^4/4$-close to $H_S(n_i,m_i,3)$ or $H_D(n_i,m_i,3)$, and\\
(b) there exists a perfect fractional matching in $H[R'^i]$.
\medskip

\pf
For each $T\in \binom{V(H)}{\le 2}$, let $\Deg^i(T):= |N_H(T) \cap \binom{R'^i}{4-|T|}|$.
By definition of $H$, we have that
\begin{itemize}
\item for any $v_j\in \mathcal{V}$, $d_H(v_j)\ge f(n',m',3) - (\gamma n') \binom{n'}{2} \ge f(n,m,3)  - \gamma n^3$, and
\item for any $T = \{v_j,u\}$ with $v_j\in \mathcal{V}$ and $u\in [n]$, $$d_H(T)=d_{F_j}(u) \le \binom{n'-1}{2} - \binom{n'-1-3(m'-1)}{2} \le\binom{n-1}{2} - \binom{n-1-3(m-1)}{2} +\gamma n^2.$$
\end{itemize}
	
Assume that $\bigwedge_i (A_i\wedge B_i \wedge C_i)$ holds.
Then $n_i = (1+o(1)) n^{0.1}$ and $m_i = (1+o(1))mn^{-0.9}$.
Since $R^i\setminus R'^i = o(n_i)$, for each $T\in \binom{V(R'^i)}{t}$ with $t\in[2]$, we have
	\begin{align*}
		\mathbb{E}[\Deg^i(T)] = (1+o(1))d_H(T)(n^{-0.9})^{4-t}.
	\end{align*}
Thus, for any $v\in \mathcal{V}\cap R^i$,
\begin{align*}
\mathbb{E}[\Deg^i(v)]\ge (1+o(1)) (f(n,m,3) -\gamma n^3)(n^{-0.9})^3 \ge f(n_i, m_i, 3) -2\gamma n_i^3,
\end{align*}
and, for any $T= \{u,v\}$ with $v\in \mathcal{V}$ and $u\in [n]$, $\mathbb{E}[\Deg^i(T)]$ is at most
\begin{align*}
(1+o(1))\left[\binom{n-1}{2} - \binom{n-1-3(m-1)}{2} +\gamma n^2\right](n^{-0.9})^2\le \binom{n_i-1}{2} - \binom{n_i-1-3(m_i-1)}{2} +2\gamma n_i^2.
\end{align*}
We apply Janson's Inequality (Theorem 8.7.2 in  \cite{pr_book}) to bound the deviation of $\Deg^i(T)$ for $|T|\leq 2$.
Write $\Deg^i(T) = \sum_{e\in N_H(T)}X_e$, where $X_e = 1$ if $e\subseteq R'^i$ and $X_e = 0$ otherwise.
Let $t = |T|\in \{1,2\}$ and $p= n^{-0.9}$.
Then
\begin{align*}
\Delta^* &=\sum_{e_i\cap e_j \ne \emptyset, ~e_i,e_j\in \binom{V(H)}{4-t}} \mathbb{P}(X_{e_i} = X_{e_j} = 1) \le \sum_{\ell=1}^{4-t} p^{2(4-t)-\ell}\binom{n-t}{4-t}\binom{4-t}{\ell}\binom{n-4}{4-t-\ell}=O(n^{0.1(2(4-t) -1)}).
\end{align*}
By Janson's inequality, for $v\in \mathcal{V}\cap R^i$,
\begin{align*}
\mathbb{P}(\Deg^i(v) \le (1-\gamma)\mathbb{E}[\Deg^i(v)]) \le e^{-\gamma^2 \mathbb{E}[\Deg^i(v)]/(2+\Delta^* /\mathbb{E}[\Deg^i(v)]) }
\le e^{-\Omega (n^{0.3}/(2+n^{0.5}/n^{0.3}))}
= e^{-\Omega(n^{0.1})},
\end{align*}
and, for the pair $\{v,u\}$ with $v\in \mathcal{V}$ and $u\in [n]$ (by considering the complement of $H$), we can have
	\begin{align*}
	\mathbb{P}(\Deg^i(\{v,u\}) \ge (1+\gamma)\mathbb{E}[\Deg^i(\{v,u\})])\le e^{-\Omega(n^{0.1})}
	\end{align*}
By union bound, with probability $1- o(1)$ we derive from above that for all $1\le i \le n^{1.1}$
\begin{itemize}
\item [1).] for any $v\in \mathcal{V} \cap R^i$, $\Deg^i(v)\ge (1-\gamma)\mathbb{E}[\Deg^i(v)])\geq f(n_i, m_i, 3) -3\gamma n_i^3$, and
\item [2).] for any pair $\{u,v_j\} \subseteq R'^i$ with $v_j\in \mathcal{V}$ and $u\in [n]$, $$\Deg^i(\{u,v_j\}) \le \binom{n_i-1}{2} - \binom{n_i-1-3(m_i-1)}{2} +3\gamma n_i^2\leq \binom{n_i-1}{2}-\Omega(n_i^2),$$
which implies that $F_j[R'^i\cap [n]]$ is not $\epsilon^3/2$-close to $S(n_i,m_i,3)$, since $m_i = (1+o(1))mn^{0.9}$ and $m<(1-c)n/3$.
\end{itemize}	
This shows that $H[R'^i \setminus \mathcal{U}]$ is not $\epsilon^4/4$-close to $H_S(n_i,m_i,3)$, where $\gamma \ll \epsilon$.

Let $\mathcal{V}_0:= \{ v_i\in \mathcal{V} : ~F_i[[n]]$  is not $\epsilon$-close to $D(n,m,3)\}$.
We claim that $|\mathcal{V}_0|> \epsilon n$.	
Otherwise $|\mathcal{V}_0| \le \epsilon n$, then we have
\begin{align*}
|E(H_D(n',m',3))\setminus E(H(\mathcal{F}))|\le \epsilon n \binom{n}{3} + (m-\epsilon n) \epsilon n^3 + \gamma (n')^4\le \epsilon (n')^4,
\end{align*}
a contradiction as $H(\mathcal{F})$ is not $\epsilon$-close to $H_D(n',m',3))$.
As $|\mathcal{V}_0|> \epsilon n$, with probability $1-o(1)$ we have (using Lemma~\ref{chernoff}) that
\begin{itemize}
\item [3).] $|R'^i \cap  \mathcal{V}_0| \ge \frac{\epsilon n_i}{2}$ for all $1\le i\le n^{1.1}$.
\end{itemize}
	
For $v_j\in R'^i \cap  \mathcal{V}_0$, we consider $F_j[[n]]$.
Let $G$ be the complement of $F_j[[n]]$.
Then for any $S\subseteq V(G)$ with $|S|> 3m-\epsilon n$, we have $e(G[S])\ge \epsilon e(G)$.
Since otherwise $|E(D(n,m,3))\backslash E(F_j[[n]])|\le \epsilon n \binom{n}{2} +\epsilon e(G) < \epsilon n^3$, contradicting $v_j\in \mathcal{V}_0$.
By Lemma~\ref{subgrpah indenpent set}, the maximum size of the complete $3$-graph in $F_j[R^i \cap [n]]$ is no more than
$(3m/n-\epsilon +\gamma)n^{0.1}\leq 3m_i-\epsilon n_i/2$ with probability at least $1-(n^{O(1)}e^{-\Omega(n^{0.1})})$.
Assuming $\bigwedge_i (A_i\wedge B_i \wedge C_i)$, this implies that $F_j[R'^i\cap [n]]$ is not $\epsilon^3/2$-close to $D(n_i,m_i,3)$.
By union bound, with probability $1-o(1)$, we have
\begin{itemize}
\item [4).] for all $1\le i\le n^{1.1}$ and $v_j\in R'^i \cap  \mathcal{V}_0$, $F_j[R'^i\cap [n]]$ is not $\epsilon^3/2$-closed to $D(n_i,m_i,3)$.
\end{itemize}
By  3) and 4), we see that,  with probability $1-o(1)$, $H[R'^i \setminus \mathcal{U}]$ is not $\epsilon^4/4$-close to $H_D(n_i,m_i,3)$, proving part (a) of Claim C.
\medskip
	
It remains to show part (b) of Claim C, that is, to construct a perfect fractional matching $w_i$ in $H[R'^i]$ for each $1 \le i\le n^{1.1}$.
Our main tool is the stability result, Lemma~\ref{stability}.
	
Fix some $1\le i\le n^{1.1}$. We write $R'^i \cap [n] =\{x_1^i,...,x_{n_i}^i \} $ with $x_1^i<x_2^i<...<x_{n_i}^i$
and define $[d]_i : = \{x_1^i,x_2^i,...,x_d^i \}$ for any integer $d$.
We now state two simple inequalities for later use:
\begin{align}\label{equ:f(xy)}
f(x,y,3)\ge f(x,y-a,3) + \binom{a}{3} \mbox{ ~and~ } f(x,y,3)\ge f(x,y+a,3) -3ax^2
\end{align}
hold for any positive integers $x,y,a$ with $a<y$.

To construct a perfect fractional matching $w_i$ in $H[R'^i]$,
first we consider $v_j\in R'^i\cap \mathcal{V}_0$ and assign weights to the edges of $H[R'^i]$ containing $v_j$.
Using  1), and by \eqref{equ:f(xy)} and the fact that $\gamma \ll \epsilon \ll1$,
$$|F_j[R'^i\cap [n]]|=\Deg^i(v_j)\ge f(n_i, m_i, 3) -3\gamma n_i^3\ge f(n_i, m_i +\epsilon^{20}n_i, 3) -\epsilon^{16} n_i^3.$$
By 2) and 4), $F_j[R'^i\cap [n]]$ is not $\epsilon^3/2$-close to $S(n_i,m_i,3)$ or $D(n_i,m_i,3)$.
Since $|E(S(n_i,m_i+\epsilon^{20}n_i,3)) \setminus E(S(n_i,m_i,3))|\le \epsilon^{20}n_i^3$ and $|E(D(n_i,m_i+\epsilon^{20}n_i,3))\setminus E(D(n_i,m_i,3))|\le 3\epsilon^{20}n_i^3$,
we see that $F_j[R'^i\cap [n]]$ is not $\epsilon^4$-close to $S(n_i,m_i+\epsilon^{20}n_i,3 )$ or $D(n_i,m_i+\epsilon^{20}n_i,3)$.
Then by Lemma~\ref{stability} and the fact that $F_j$ is stable, $F_j[R'^i\cap [n]]$ contains a matching $M_j$ with $V(M_j) = [3m_i+3\epsilon^{20}n_i]_i$.
Now we assign weights $w_i(e)$ to all edges $e$ of $H[R'^i]$ with $v_j\in e$ as follows:
If $e\setminus v_j\in M_j$, then let $w_i(e) = \frac{1}{m_i+\epsilon^{20}n_i}$, and otherwise let $w_i(e) = 0$.
	
Next we consider $v_j\in R'^i\cap (\mathcal{V}\backslash \mathcal{V}_0)$.
By  1) and \eqref{equ:f(xy)}, we have $$|F_j[R'^i\cap [n]]|\ge f(n_i, m_i, 3) -3\gamma n_i^3 \ge f(n_i, m_i-6\gamma^{\frac{1}{3}} n_i, 3).$$
By Theorem~\ref{matching} and the fact that $F_j$ is stable, $F_j[R'^i\cap [n]]$ contains a matching $M_j$ with $V(M_j) = [3m_i-18\gamma^{\frac{1}{3}} n_i]_i $.
Then we assign weights $w_i(e)$ to all edges $e$ of $H[R'^i]$ with $v_j\in e$ as follows:
If $e\setminus v_j \in M_j$, then let $w_i(e) = \frac{1}{m_i-6\gamma^{1/3} n_i}\ $; and otherwise let $w_i(e) = 0$.

Note that for every $v_j\in R'^i\cap \mathcal{V}$, we have defined weights $w_i(e)$ for all edges $e\in H[R'^i]$ with $v_j\in e$, whose total weights equal one.
In the remaining proof, we want to extend this function $w_i$ to
entire $H[R'^i]$ to form a perfect fractional  matching.
We complete this in two steps.

First, we define a perfect fractional matching $w$ (as the {\it projection} of $w_i$) in the complete 3-graph $K$ on vertex set $R'^i\cap [n]$.
Note that a function $w: E(K)\to [0,1]$ is a perfect fractional matching if and only if $w(v):=\sum_{v\in f\in K}w(f)=1$ holds for every $v\in V(K)$.
Initially, we define a function $w': E(K)\to [0,1]$ such that, for each $f\in E(K)$, $w'(f):=\sum_e w_i(e)$ over all edges $e\in H[R'^i]$ with $f\subseteq e$ and $|e\cap \mathcal{V}|=1$.
Since $|\mathcal{V}_0|>\epsilon n$ and $\gamma \ll \epsilon$, it follows from the above definitions on $w_i$ that for any $v\in R'^i\cap [n]$,
$$w'(v):=\sum_{v\in f\in K} w'(f)\leq\frac{|\mathcal{V}_0|}{m_i+\epsilon^{20}n_i} + \frac{m_i-|\mathcal{V}_0| }{m_i-6\gamma^{\frac{1}{3}} n_i}\leq \frac{\epsilon n_i}{m_i+\epsilon^{20}n_i} + \frac{m_i -\epsilon n_i}{m_i-6\gamma^{\frac{1}{3}} n_i} <1.$$
Since $\epsilon \ll c$, we have $3m_i+3\epsilon^{20}n_i < n_i -4$.
So there exists a vertex set $\{a_1,a_2,a_3,a_4\}$ in $K$ such that $w'(a_i) = 0$, for $i\in [4]$.
Let $K'$ be the 3-graph obtained from K by deleting vertices $a_1,a_2,a_3,a_4$. Starting with $w:=w'$, we increase $w$ using the following iterations:
(i) pick a vertex $v$ in $V(K')$ with maximum $w(v)$;\footnote{Note that this maximum $w(v)$ is strictly less than 1.}
(ii) pick any edge $f\in K'$ containing $v$ and update $w(f)\leftarrow w(f)+1-w(v)$;
(iii) delete all vertices $u\in V(K')$ with $w(u)=1$ (which must include the vertex $v$) from $K'$;
(iv) if $|V(K')| \le 2$, then terminate; otherwise go to (i) again.
This must terminate in finitely many iterations and when it terminates, we obtain a fractional matching $w$ in $K$ such that $w(a_i) = 0$ for $i\in [4]$ and $|V(K')| \le 2$.
So there exist two vertices $b_1,b_2$ in $V(K)\setminus \{a_1,a_2,a_3,a_4\}$ such that for any vertex $v$ in $V(K)\setminus \{a_1,a_2,a_3,a_4,b_1,b_2\}$, $w(v) =1$.
We may suppose $1\ge w(b_1) \ge w(b_2)$.
Let $w(a_1,a_2,b_1) = 1-w(b_1)$, $w(a_1,a_2,b_2) = \frac{w(b_1)-w(b_2)}{2} $, $w(a_3,a_4,b_2) = 1-w(b_1) + \frac{w(b_1)-w(b_2)}{2}$, and $w(a_1,a_2,a_3) = w(a_1,a_2,a_4) =w(a_1,a_3,a_4) =w(a_2,a_3,a_4) =\frac{w(b_1)+w(b_2)}{6}$.
It is easy to check that $w$ is a perfect fractional matching in $K$.

Now we notice that $\sum_{f\in K} w'(f)=\sum_{\{e\in H[R'^i]: |e\cap \mathcal{V}|=1 \}} w_i(e)=|R'^i\cap \mathcal{V}|$
and, $\sum_{f\in K} w(f)=\frac{|R'^i\cap [n]|}{3}=|R'^i\cap (\mathcal{V}\cup \mathcal{U})|$.
Moreover, the neighborhood of any $u_j\in R'^i\cap\mathcal{U}$ in $H[R'^i]$ is the complete 3-graph $K$.
So we can partition the total weight $\sum_{f\in K}( w(f)-w'(f))=|R'^i\cap \mathcal{U}|$ into $|R'^i\cap \mathcal{U}|$ copies of 1's
(say each is represented by a set $E_j$ of edges in $K$),
and then for each $u_j\in R'^i\cap \mathcal{U}$, we assign the weight of each $f\in E_j$ to be $w_i(f\cup \{u_j\})$.
One can easily check that we obtain a perfect fractional matching $w_i$ in $H[R'^i]$.
This completes the proof of Claim~C.	
\qed

\medskip
  	
From Claims~B and C, we see that the sets $R'^i$ for $1\leq i\leq n^{1.1}$ satisfy (a)-(d) in Lemma~\ref{spaning graph}.
Then by Lemma~\ref{spaning graph}, there exists a spanning subgraph $H'$ of $H$ such that for each $v\in V(H)$, $d_{H'}(v) = (1+o(1))n^{0.2}$, and $\Delta_2(H')\le n^{0.1}$.
By Theorem~\ref{regular}, $H$ contains a matching $M_b$ such that $S = V(H)\setminus V(M_b)$ contains at most $\gamma' n'$ vertices.
Since $|S\cup M_a\cup M_b| =n'=3r'+3m'+s$ where $0\leq s\leq 2$,
we can delete at most $s$ elements from $S$ to get a subset $S'$ such that $3|S'\cap (\mathcal{V}'\cup \mathcal{U}')| = |S'\cap [n']|$.
By the setting at the beginning of the proof, Lemma~\ref{absorb} assures that $H^{*}(\mathcal{F})[V(M_a)\cup S']$ has a perfect matching,
which together with $M_b$ form a matching in $H^{*}(\mathcal{F})$ of size $r'+m'$.
Equivalently, this says that $\mathcal{F}$ admits a rainbow matching, finishing the proof of Lemma~\ref{lem:not close}.	
\end{proof}

\section{Proof of Theorem~\ref{main}}
Let $n$ be a sufficiently large integer.
Let $m$ be a positive integer with $n\geq 3m$ and let $\mathcal{F}= \{F_1, ... ,F_m\}$ be a family of $3$-graphs on the same vertex set $[n]$, such that $|F_i|> f(n,m,3)$ for each $i\in [m]$.
Suppose to the contrary that $\mathcal{F}$ does not admit a rainbow matching.
In view of Lemma~\ref{lem:stable}, we may assume that $\mathcal{F}$ is stable.
Then by Lemma~\ref{large-mat}, there exists an absolute constant
$c=c(3)>0$ such that $m\le  (1-c)n/3$.
By Theorem~\ref{n>k^2}, $m\ge n/27$.
Hence,
\begin{align}\label{equ:m}
n/27 \le m\le (1-c)n/3.
\end{align}

We now apply the following algorithm.
Initially, let $\mathcal{F}_0=\mathcal{F}$, $n_0=n$ and $m_0=m$.
We repeat the following iterations.
Suppose that we have defined $\mathcal{F}_i$, which contains $m_i$ 3-graphs on the same vertex set $[n_i]$.
\begin{itemize}
\item Step 1: Apply Corollary~\ref{coro:stable-saturated} to ${\cal F}_{i}$, we obtain a family
  ${\cal F}_{i+1}$ of $3$-graphs on the vertex set $[n_i]$  that is
  both stable and saturated, and set $n_{i+1}=n_i$ and $m_{i+1}=m_i$.

\item Step 2: If for any $F\in {\cal F}_{i+1}$ and any $v\in [n_{i+1}]$, $d_F(v)< \binom{n_{i+1}-1}{2}$, then set $t:=i+1$ and output ${\cal F}_{t},n_t,m_t$.

\item Step 3: If there exist $F\in \mathcal{F}_{i+1}$ and $v\in [n_{i+1}]$ such that $d_F(v)= \binom{n_{i+1}-1}{2}$,
then set $n_{i+1}'=n_{i+1}-1$, $m_{i+1}'=m_{i+1}-1$, and
$\mathcal{F}'_{i+1}:=\{F'-v: F'\in {\cal F}_i \setminus
\{F\}\}$. Relabel the vertices if necessary so that all 3-graphs in
${\cal F}'_{i+1}$ have the same vertex set $[n'_{i+1}]$.
Set ${\cal F}_{i}:={\cal F}'_{i+1}, n_i:= n'_{i+1}, m_i:=m'_{i+1}$ and go to Step 1.
\end{itemize}

Let $\mathcal{F}_t$ be the resulting family of 3-graphs, which
contains $m_t$ 3-graphs on the same vertex set $[n_t]$ and admits no
rainbow matching. 
By \eqref{equ:m}, we see that $n_t\geq n-m> cn$ is sufficiently large.
We also see from Lemma~\ref{lem:1} that
$$|F|> f(n_t,m_t,3) \mbox{ holds for any } F\in \mathcal{F}_t.$$

By definition, we see that $\mathcal{F}_t$ is stable and saturated such that for any $F\in \mathcal{F}_t$ and $v\in V_t$, $d_F(v)<\binom{n_t-1}{2}$.
On the other hand, by Lemma~\ref{claim:Delta}, it further holds that
$$d_{F}(v)\leq \binom{n_t-1}{2} - \binom{n_t-1-3(m_t-1)}{2} \mbox{ for any } F\in \mathcal{F}_t \mbox{ and } v\in V_t.$$
Since $n_t$ is sufficiently large, using Lemma~\ref{large-mat} and Theorem~\ref{n>k^2} again, we may assume that
$$n_t/27 \leq m_t\leq (1-c)n_t/3.$$

Now we choose $0<\epsilon \ll c$.
Since $\mathcal{F}_t$ satisfies the above properties,
by applying Lemmas~\ref{lemma:H_S(n,m)},~\ref{lemma:H_D(n,m)}
and~\ref{lem:not close}, we can conclude that $\mathcal{F}_t$ admits a rainbow matching.
This is a contradiction, completing the proof of Theorem~\ref{main}. \qed

\bigskip

\noindent{\bf Acknowledgements.} The authors would like to thank Peter Frankl for bringing the references \cite{AF,F87}.
The authors are grateful to the referees for their careful reading and helpful suggestions.


\begin{thebibliography}{99}
\addtolength{\baselineskip}{-1ex}

\bibitem{ADP}
R. Aharoni and D. Howard, Size conditions for the existence of rainbow matchings, preprint.

\bibitem{AF}
J. Akiyama and P. Frankl, On the size of graphs with complete-factors, \emph{J. Graph Theory}, \textbf{9(1)} (1985), 197-201.


\bibitem{AFH12}
N. Alon, P. Frankl, H. Huang, V. R\"{o}dl, A. Rucinski and B. Sudakov, Large matchings in uniform hypergraphs and the conjecture of Erd\H{o}s and Samuels, \emph{J. Comin. Theory Ser. A}, \textbf{119} (2012), 1200-1215.

\bibitem{pr_book}
N. Alon and J. Spencer, The Probabilistic Method, Wiley-Intersci. Ser. Discrete Math. Optim., John Wiley Sons, Hoboken, NJ, (2000), third edition, (2008).

\bibitem{BDE76}
B. Bollob\'as, D. E. Daykin and P. Erd\H{o}s, \newblock{Sets of independent edges of a hypergraph}, \emph{Quart. J. Math. Oxford Ser.} (2), \textbf{27} (1976), 25-32.

\bibitem{Erdos65}P. Erd\H{o}s,   A problem on independent $r$-tuples, \emph{Ann. Univ. Sci. Budapest. E\"otv\"os Sect. Math.}, \textbf{8} (1965),
93-95.

\bibitem{F87}
P. Frankl, The shifting technique in extremal set theory, \emph{Surveys in combinatorics} 1987 (New Cross, 1987), 81-110, London Math. Soc. Lecture Note Ser., 123, Cambridge Univ. Press, Cambridge, 1987.

\bibitem{Fr13} P. Frankl, Improved bounds for Erd\H{o}s matching conjecture, {\it J. Combin. Theory Ser. A}, {\bf 120} (2013), 1068-1072.



\bibitem{F17}
P. Frankl, On maximum number of edges in a hypergraph with given matching number, \emph{Discrete Appl. Math.}, \textbf{216} (2017), 562-581.

\bibitem{Fr172} P. Frankl,
Proof of the Erd\H{o}s matching conjecture in a new range,
\emph{Israel J. Math.}, \textbf{222} (2017), 421-430.

\bibitem{FLM12}
P. Frankl, T. {\L}uczak and K. Mieczkowska, On matchings in hypergraphs, \emph{Electron. J. Combin.}, \textbf{19} (2012), \#R42.

\bibitem {FK18}
P. Frankl and A. Kupavskii, The Erd\H{o}s matching conjecture and concentration inequalities, arXiv:1806.08855.

\bibitem {FK20}
P. Frankl and A. Kupavskii, Simple juntas for shifted families, Discrete Analysis 2020:14, 18 pp. arXiv:1901.03816.

\bibitem{FR85}
P. Frankl and V. R\"odl, Near perfect covering in graphs and hypergraphs, \emph{European J. Combin.}, \textbf{6(4)} (1985), 317-326.


\bibitem{HLS12}
H. Huang, P. Loh and B. Sudakov, The size of a hypergraph and its matching number, \emph{Combin. Probab. Comput.}, \textbf{21} (2012), 442-450.

\bibitem{JK20}
F. Joos and J. Kim, On a rainbow version of Diracs theorem, \emph{Bull. London Math. Soc.,} \textbf{52} (2020) 498-504.

\bibitem{KL17}
N. Keller and N. Lifshitz, The Junta method for hypergraphs and the Erd\H{o}s-Chv\'atal simplex conjecture, arXiv:1707.02643.

\bibitem{KK20} S. Kiselev and A. Kupavskii, Rainbow matchings in $k$-partite hypergraphs, \emph{Bull. London Math. Soc.}, https://doi.org/10.1112/blms.12423

\bibitem{LYY19}
H. Lu, X. Yu and X. Yuan, Nearly perfect matchings in uniform hypergraphs, arXiv:1911.07431.

\bibitem{LYY20} H. Lu, X. Yu and X. Yuan, Rainbow matchings for 3-uniform hypergraphs, arXiv:2004.12558.


\bibitem{LWY20}
H. Lu, Y. Wang and X. Yu, A better bound on the size of rainbow matchings, arXiv:2004.12561.


\bibitem{LM14}
T. {\L}uczak and K. Mieczkowska, On Erd\H{o}s' extremal  problem on matchings in hypergraphs, \emph{J. Combin. Theory Ser. A} \textbf{124} (2014), 178-194.

\bibitem{Py86}
L. Pyber,  A new generalization of the Erd\H{o}s-Ko-Rado theorem, \emph{J. Combin. Theory Ser. A}, \textbf{43} (1986),
85-90.

\end{thebibliography}
\end{document}